\documentclass[reqno]{amsart}
\usepackage{hyperref}

\def\N{{\mathbb{N}}}

\def\R{{\mathbb{R}}}

\usepackage{amssymb}
\usepackage{xcolor}
\DeclareUnicodeCharacter{2264}{\leq}
\DeclareUnicodeCharacter{2266}{\leqq}

\theoremstyle{plain}
\newtheorem{theorem}{Theorem}
\newtheorem{proposition}{Proposition}
\newtheorem{definition}{Definition}
\newtheorem{lemma}{Lemma} 
\newtheorem{corollary}{Corollary}
\theoremstyle{remark}
\newtheorem{remark}{Remark}

\newtheorem{Exemps}{Example}

\title[Minimax theorem]{Some new minimax theorems for generalized convexity}
\author{Mohammed Bachir}

\begin{document}

\date{\today} 
\subjclass{Primary 46E15, 46E10, 26B25, Secondary 26D99}
\address{Laboratoire SAMM 4543, Universit\'e Paris 1 Panth\'eon-Sorbonne, France}

\email{Mohammed.Bachir@univ-paris1.fr}
\begin{abstract}
The aim of this article is to establish new two-functions minimax inequalities extending classical results such as Simons' minimax  theorem. Our results will be proved in a non-compact setting. We also prove, under  general conditions, that the one-function minimax equality is in fact equivalent to the well known Simons inequality. Some applications will be given.
\end{abstract}
\maketitle
{\bf Keywords:} Minimax theorem, Alternative theorem, convexlikeness, infsup-convexity, function spaces. 


\section{Introduction}

Let $X$ and $Y$ be arbitrary nonempty sets and $f: X\times Y\to \R$ be a function.  Recall that, $f$ is said to be $t$-convexlike on $X$ for some $t\in (0,1)$, if and only if for all $x_1, x_2\in X$, there exists $x_3\in X$ such that $f(x_3,y)\leq t f(x_1,y) + (1-t) f(x_2,y)$, for all $y\in Y$ and  $f$ is said to be $t$-concavelike on $Y$, if and only if for all $y_1, y_2\in Y$, there exists $y_3\in Y$ such that $f(x, y_3)\geq t f(x,y_1) + (1-t) f(x,y_2)$, for all $x\in X.$ We say that $f$ is convexlike on $X$, a concept due to Fan, (resp. concavelike on $Y$) if and only if $f$ is $t$-convexlike on $X$ (resp. $f$ is $t$-concavelike on $Y$) for every $t\in (0,1)$. The more general notions of infsup-convexity and supinf-concavity will be given in Definition \ref{def01} and Proposition \ref{infsup}. Several properties of these concepts and their uses can be found in  \cite{Simons1, St, Ps, RG0, RG1, RG2}. We say that $f$ has a property $(P)$ on $X$ (resp. on $Y$) if the function $f(\cdot, y)$ (resp. $f(x, \cdot)$) has the property $(P)$ for every $y\in Y$ (resp. for every $x\in X$). In general, the properties that concern us in this paper will be $(P)=$``boundedness", ``semicontinuity", ``convexity" or ``generalized convexity". Clearly, a convex (resp. concave) function $f$ on a convex set $X$ is convexlike (resp. concavelike) on $X$ but the converse is not true in general.

The Fan minimax theorem in \cite[Theorem 2]{Fk} (see also \cite[Theorem 11]{Simons1}) says that if $Y$ is an arbitrary nonempty set, $X$ is a compact Hausdorff space and $f: X\times Y\to \R$ is a function, lower semicontinuous convexlike on $X$ and concavelike on $Y$, then
$$\inf_{x\in X} \sup_{y\in Y} f(x,y) = \sup_{x\in Y}\inf_{y\in X} f(x,y). \hspace{3mm} (\bullet)$$
The Fan minimax theorem was extended by Simons in \cite[Theorem 5c]{Simons0} (see also \cite[Theorem 26]{Simons1}) to two-functions minimax inequality as follows. Let $Y$  be an arbitrary nonempty sets, $X$ is a compact Hausdorff  space, $t\in (0,1)$ and $f, g: X\times Y\to \R$ be two functions such that:

$(i)$ $f$ is $t$-convexlike and lower semicontinuous on $X$,

$(ii)$ $g$ is $t$-concavelike on $Y$,

$(iii)$ $f\leq g$.

Then, 
$$\inf_{x\in X} \sup_{y\in Y} f(x,y) \leq \sup_{y\in Y}\inf_{x\in X} g(x,y).$$
The aim of this article is to establish new two-functions minimax inequalities extending classical results such as Fan, K$\ddot{\textnormal{o}}$nig and Simons  theorems. Our results will be derived from a new theorem of the alternative (Theorem ~\ref{theorem1}).
\vskip5mm
\paragraph{ \bf A. Minimax theorem with Simons-like inequality.} In a first result (Theorem \ref{km1}), we will replace the compactness of $X$ and the semicontinuity of $f$ on $X$ by the following more general condition (in arbitrary nonempty sets $X$ and $Y$), which we will call the {\it Simons-like inequality}: for every net $(x_\alpha)_{\alpha\in I}\subset X$, 
\[
\inf_{x\in X} \sup_{y\in Y} f(x,y) \leq \sup_{ y\in Y} \underset{\alpha}{\limsup} f(x_\alpha,y). \hspace{3mm} (\bullet \bullet)
\]
This inequality  was inspired by \cite[Lemma, P. 704]{Simons3} (see also \cite[Corollary, P. 96]{DF}) and  justified by the inequality obtained in Lemma \ref{corollary0}. The Simons-like inequality is clearly satisfied when $X$ is Hausdorff compact set and $f$ lower semicontinuous on $X$, which allows us to recover immediately the Simons minimax inequality cited above. We will show (Corollary \ref{km2}) that the Simons-like inegality in $(\bullet \bullet)$ is in fact a necessary and sufficient condition for obtaining the one-function minimax equality in $(\bullet)$, whenever $f$ is assumed to be bounded $t$-convexlike on $X$ and $s$-concavelike on $Y$ for some $t,s\in (0,1)$. Simple examples where the Simons-like inequality holds without compactness are given in Example \ref{Exp}, thus ensuring examples of non-compact minimax theorems in Example \ref{Exp1}.

The classical Simons inequality was first introduced by Simons in \cite{Simons3} under certain conditions and then generalised by Deville and Finet in \cite{DF} and Kivisoo and Oja in \cite{KO} (see also \cite[Theorem 10.5]{COR}). Note here that the authors just cited have considered this inequality with sequences $(x_n)_n$ instead of nets $(x_\alpha)_{\alpha\in I}$ (see Remark ~\ref{remsuite}). By combining our result in Theorem \ref{km1} together with the main result in \cite{KO} about the Simons inequality, we give in Corollary \ref{corder} a new minimax inequality without the compactness or semicontinuity assumptions but assuming only convexity on one variable.

\paragraph{ \bf B. Minimax theorem in the pseudocompact framework.} The second part of the paper deals with the two-functions minimax inequality in the context of completely regular Hausdorff pseudocompact spaces. Recall that a topological space is said to be pseudocompact if its image under any real-valued continuous function is bounded. Every countably compact space is pseudocompact. For various properties on pseudocompact spaces we refer to the article by Stephenson in \cite{Strm}.  

Let $X$ be an arbitrary nonempty set, $Y$ be a completely regular Hausdorff pseudocompact space and $f,g :X\times Y \to \R$ be two functions such that:

 $(i)$ $f$ is bounded on $X\times Y$, infsup-convex on $X$ and the family $\lbrace f(x,\cdot): x\in X\rbrace$ is equicontinuous on $Y$,

 $(ii)$ $g$ is supinf-concave on $Y$ and bounded on $X$,

$(iii)$ $f\leq g$.

 Then, it is shown in theorem \ref{cor-equic} that the two-functions minimax inequality holds: $\inf_{x\in X} \sup_{y\in Y} f(x,y) \leq \sup_{y\in Y}\inf_{x\in X} g(x,y).$

To the best of our knowledge, this situation (in the pseudocompact framework) is also new in the literature. Moreover, the notion of infsup-convexity is weaker than convexity, convexlikeness, $t$-convexlikeness etc. As a consequence, we prove (Corollary \ref{let})  that if $Z$ is a Hausdorff countably compact space, $E$ is a Banach space and $(f_n)_n$ is a uniformly bounded sequence of continuous functions from $Z$ into $E$ which $w^*$-pointwize converges to $0$ on $Z$, then there exists a sequence of linear convex combinations of $(f_n)$ which is uniformly convergent to $0$ on $Z$ (an extension of a known result in  \cite[Proposition, p. 103]{DF}).

\vskip5mm
This paper is organized as follows. In section ~\ref{RS2}, we recall the definitions of infsup-convexity and supinf-concavity and we prove that these concepts generalise those of $t$-convexlikness and $t$-concavelikeness. In Section ~\ref{RS3} we prove a new alternative theorem (Theorem ~\ref{theorem1}). This result will be used  to prove our results on minimax theorems. In Section ~\ref{RS4}, we prove our main results Theorem \ref{km1}, Corollary \ref{km2} and Theorem \ref{cor-equic} and we give some consequences and examples. In Section \ref{RS5}, we give an application to $w^*$-pointwise convergent sequence of vector-valued functions.
\section{infsup-convexity and supinf-concavity} \label{RS2}
 The concept of the infsup-convexity  is a  well-known generalization of convexity, convexlikeness and $t$-convexlikeness ($t\in (0,1)$) used by several authors. It was considered for the ﬁrst time by Stefanescu in \cite[Deﬁnition 2.11]{St} as aﬃne weakly convexlikeness, the present nomenclature is due to Ruiz Gal$\acute{\textnormal{a}}$n and was given in \cite[Definition 2.1]{RG1} (see also \cite{RG0, RG2}). For $n\geq 1$, the $(n-1)$-dimentional simplex is defined by $$\Delta_n:=\lbrace (\lambda_1,...,\lambda_n)\in \R_+^n: \sum_{i=1}^n \lambda_i =1 \rbrace.$$
\begin{definition} \label{def01} Let $X$ and $Y$ be nonempty sets and $f: X\times Y\to \R$ be a function. 

$(i)$ We say that $f$ is infsup-convex on $X$ if and only if,
\begin{eqnarray*}
\inf_{x\in X}\sup_{y\in Y} f(x, y) &=& \inf_{\underset{\underset{(\lambda_1,\lambda_2,...,\lambda_n)\in \Delta_n}{ x_1, ...,x_n\in X}}{n\geq 1}} \sup_{y\in Y}\sum_{i=1}^n \lambda_i f(x_i, y)\\
&=& \inf_{\Psi\in \textnormal{conv}(C_X)} \underset{y\in Y}{\sup} \Psi(y)
\end{eqnarray*}
where $C_X:=\lbrace f(x,\cdot): x\in X\rbrace$.

$(ii)$ We say that $f$ is supinf-concave on $Y$ if and only if,
\begin{eqnarray*} 
\sup_{y\in Y}\inf_{x\in X} f(x, y) &=& \sup_{\underset{\underset{(\lambda_1,\lambda_2,...,\lambda_n)\in \Delta_n}{ y_1, ...,y_n\in Y}}{n\geq 1}} \inf_{x\in X}\sum_{i=1}^n \lambda_i f(x, y_i)\\
&=& \sup_{\Psi\in \textnormal{conv}(C_Y)} \underset{x\in X}{\inf} \Psi(x),
\end{eqnarray*}
where $C_Y:=\lbrace f(\cdot, y): y\in Y\rbrace$
\end{definition}
We easily see from the definitions that:
\begin{eqnarray*}
 f \textnormal{ is convexlike on } X 
&\Longrightarrow & f \textnormal{ is infsup-convex on } X.
\end{eqnarray*}
The converse is not true in general (see for instance \cite[Example 2.15]{St} see also \cite{RG2}). However, the notion of infsup-convexity (resp. supinf-concavity) is also a generalisation of $t$-convexity (resp. $t$-concavity) with an arbitrarily fixed $t\in (0,1)$. This follows from a density argument by using the following known fact: 

\noindent {\bf Fact 1.} (see  \cite{Ps} and \cite[Remark 2.8]{St}) If $f$ is $t$-convexlike on $X$ for some $t\in (0,1)$, then for any integer $n\geq 1$, there exists a dense subset $D_n(t)$ of $\Delta_n$ such that: $\forall x_1,x_2,...,x_n\in X$, $\forall (a_1,a_2,...,a_n)\in D_n(t)$, $\exists x_0\in X$ such that,
\begin{eqnarray}\label{Fa} f(x_0,\cdot)\leq \sum_{i=1}^n a_i f(x_i, \cdot), \textnormal{ on } Y.
\end{eqnarray}

\begin{proposition}  \label{infsup} Let $X$ and $Y$ be nonempty sets and $f: X\times Y\to \R$ be a function.

$(i)$ Suppose that $f$ is $t$-convexlike on $X$ for some $t\in (0,1)$, bounded on $Y$ and such that $\sup_{y\in Y} \inf_{x\in X} f(x,y)>-\infty$. Then, $f$ is infsup-convex on $X$.

$(ii)$ Suppose that $f$ is $s$-concavexlike on $Y$ for some $s\in (0,1)$, bounded on $X$ and such that $\inf_{x\in X}\sup_{y\in Y}  f(x,y) <+\infty$. Then, $f$ is supinf-concave on $Y$.
\end{proposition} 

\begin{proof} We will only give the proof of $(i)$, the part $(ii)$ can be obtained in a similar way. Using Fact 1 in $(\ref{Fa})$, we get  
\begin{eqnarray} \label{V0}
\inf_{x\in X}\sup_{y\in Y} f(x, y)\leq \inf_{\underset{\underset{(a_1,a_2,...,a_n)\in D_n(t)}{ x_1, ...,x_n\in X}}{n\geq 1}} \sup_{y\in Y}\sum_{i=1}^n a_i f(x_i, y).
\end{eqnarray}

We want to establish the same inequality with $\Delta_n$ instead of $D_n(t)$. For every $\varepsilon>0$, there exists $N\geq 1$, $(\bar{\lambda}_1,...,\bar{\lambda}_N)\in \Delta_N$ and $\bar{x}_1, ...,\bar{x}_N\in X$ such that 

\begin{eqnarray}\label{V1}
\sup_{y\in Y} \sum_{i=1}^N \bar{\lambda} _i f(\bar{x}_i, y) -\varepsilon < R:=\inf_{\underset{\underset{(\lambda_1,\lambda_2,...,\lambda_n)\in \Delta_n}{ x_1, ...,x_n\in X}}{n\geq 1}} \sup_{y\in Y}\sum_{i=1}^n \lambda_i f(x_i, y).
\end{eqnarray}

Notice that $R\geq \sup_{y\in Y} \inf_{x\in X} f(x,y)>-\infty$. By the boundedness of $f$ on $Y$, let us choose  $\beta >0$ small enough such that $$ \beta\sum_{i=1}^N \sup_{y\in Y} |f(\bar{x}_i, y)| \leq 1,$$ 
By the density of the set $D_N(t)$ in $\Delta_N$, there exists $(\bar{a}_1,...,\bar{a}_N)\in D_N(t)$ such that $\max \lbrace |\bar{\lambda}_i-\bar{a}_i|: i=1,...,N\rbrace <\varepsilon \beta$. Then, for every $y\in Y$ we have,
\begin{eqnarray*}
|\sum_{i=1}^N (\bar{a}_i-\bar{\lambda}_i ) f(\bar{x}_i, y)| &\leq&\sum_{i=1}^N |\bar{a}_i-\bar{\lambda}_i| |f(\bar{x}_i, y)|\\
&\leq& \varepsilon \beta\sup_{y\in Y}\sum_{i=1}^N |f(\bar{x}_i, y)|\\
&\leq& \varepsilon.
\end{eqnarray*}
Then, we obtain
\begin{eqnarray*}
\sup_{y\in Y} \sum_{i=1}^N \bar{a}_i f(\bar{x}_i, y) \leq \sup_{y\in Y} \sum_{i=1}^N \bar{\lambda}_i f(\bar{x}_i, y)  + \varepsilon.
\end{eqnarray*}
Thus, 
\begin{eqnarray*}
\inf_{\underset{\underset{(a_1,a_2,...,a_n)\in D_n(t)}{x_1, ...,x_n\in X}}{n\geq 1}} \sup_{y\in Y} \sum_{i=1}^n a_i f(x_i, y) &\leq&  \sup_{y\in Y} \sum_{i=1}^N \bar{a}_i f(\bar{x}_i, y) \\
&\leq& \sup_{y\in Y} \sum_{i=1}^N \bar{\lambda}_i f(\bar{x}_i, y)  + \varepsilon.
\end{eqnarray*}

It follows from $(\ref{V0})$ and $(\ref{V1})$ that, for every $\varepsilon>0$
\begin{eqnarray*}
\inf_{x\in X}\sup_{y\in Y} f(x, y) &\leq&   \inf_{\underset{\underset{(\lambda_1,\lambda_2,...,\lambda_n)\in \Delta_n}{ x_1, ...,x_n\in X}}{n\geq 1}} \sup_{y\in Y}\sum_{i=1}^n \lambda_i f(x_i, y)+ 2\varepsilon.
\end{eqnarray*}
Passing to the limit when $\varepsilon \to 0$, we get 
\begin{eqnarray*}
\inf_{x\in X}\sup_{y\in Y} f(x, y) &\leq&   \inf_{\underset{\underset{(\lambda_1,\lambda_2,...,\lambda_n)\in \Delta_n}{x_1, ...,x_n\in X}}{n\geq 1}} \sup_{y\in Y}\sum_{i=1}^n \lambda_i f(x_i, y).
\end{eqnarray*}
The inverse inequality is always true. Hence, $(i)$ is proved.
\end{proof} 
\section{An alternative theorem} \label{RS3}
 There is a wealth of literature on the subjects of alternatives  theorems. Some work on the alternatives  in convex and non-convex frameworks and their applications  to optimization theory and minimax inequalities can be found in \cite{Cr1, Cr2, RG0, RG1, Je1, JON, Ld, St1}, but the list is not exhaustive. We can also found some recent work using  the so called infsup-convexity in \cite{RG0, RG1, RG2} and the references therein. The earliest version of a nonlinear alternative theorem was given by Fan, Glicksberg, and Hoffman \cite{FGH} as follows : Let $C$ be a convex set of some real vector space and $f_i; C \to \R$, $i=1,..., n$ be convex functions. Then, either the system $x\in C$, $f_i(x) <0, i=1,...,n$ has a solution or there exists $\lambda^*\in \R_+^n$, $\lambda^*\neq 0$ such that $\sum_{i=1}^n \lambda^*_i f_i(x)\geq 0$ for all $x\in C$. 

We give in the following theorem a new theorem of the alternative which deals with an arbitrary family of functions. This result  will be very useful for the rest of our work. Let's start by recalling some well-known general facts. Given a topological space $(Z,\tau)$, the space $(\mathcal{C}_b(Z),\|\cdot\|_{\infty})$ denotes the classical Banach space of all real-valued bounded continuous functions on $Z$. The continuous Dirac map is defined by $\delta : Z \longrightarrow \delta(Z)\subset B_{(C_b(Z))^*}$, $z \mapsto  \delta_z$, where $B_{(C_b(Z))^*}$ is the closed unit ball of the dual space $(C_b(Z))^*$ and $\delta_z : C_b(Z)\longrightarrow \R$ is the linear continuous map defined by $\delta_z(\varphi)=\varphi(z)$ for all $z\in Z$ and all $\varphi\in C_b(Z)$. Notice that $\|\delta_z\|\leq 1$, for all $z\in Z$. We denote by $\mathrm{conv}\left(\delta(Z)\right):=\mathrm{conv}\lbrace \delta_z: z\in Z\rbrace$ and $\overline{\mathrm{conv}}^{w^*}\left(\delta(Z)\right):=\overline{\mathrm{conv}}^{w^*}\lbrace \delta_z: z\in Z\rbrace$ the convex hull and the $w^*$-closed convex hull of $\delta(Z)$ respectively. By $\tau_p$ we denote the topology in $\mathcal{C}_b(Z)$ of pointwize convergence on $Z$. If $Z$ is an arbitrary non-empty set, we equip it with the discrete distance, so that $\mathcal{C}_b(Z)$ coincides with $\ell^{\infty}(Z)$ the space of all real-valued bounded functions. Representation theorems for the dual $(\mathcal{C}_b(Z))^*$, especially when $Z$ is a normal space, can be found in \cite[Theorem 14.10, Chapter 14]{AB}. We denote $\mathcal{C}^+_b(X)$ the closed convex positive cone of $\mathcal{C}_b(X)$ and $(\mathcal{C}^+_b(X))^*$ ils dual. The function $\boldsymbol{1}_X$ denotes the constant function equal to $1$ on $X$.

\begin{proposition} \label{pcone} Let $X$ be a topological space. Then, $$(\mathcal{C}^+_b(X))^*=\cup_{\lambda\geq 0} \lambda \overline{\textnormal{conv}}^{w^*}\lbrace \delta_x: x\in X\rbrace.$$
\end{proposition}

\begin{proof} From \cite[Section 2, p. 539]{JON}, since $\boldsymbol{1}_X \in \textnormal{int}(\mathcal{C}^+_b(X))$ (the interior of $\mathcal{C}^+_b(X)$) then 
$$\mathcal{C}^+_b(X) =\cup_{\lambda\geq 0} \lbrace \nu \in (\mathcal{C}^+_b(X))^* : \langle \nu, \boldsymbol{1}_X\rangle=1\rbrace.$$
It suffices to prove that $\lbrace \nu \in (\mathcal{C}^+_b(X))^* : \langle \nu, \boldsymbol{1}_X\rangle=1\rbrace=\overline{\textnormal{conv}}^{w^*}\lbrace \delta_x: x\in X\rbrace.$ Let $\mu \in \lbrace \nu \in (\mathcal{C}^+_b(X))^* : \langle \nu, \boldsymbol{1}_X\rangle=1\rbrace$ and suppose by contradiction that $\mu \not \in \overline{\textnormal{conv}}^{w^*}\lbrace \delta_x: x\in X\rbrace.$ By the Hahn-Banach theorem, there exists $f\in \mathcal{C}_b(X) \setminus \{0\}$ and $r\in \R$ such that 
$$\sup_{Q\in \overline{\textnormal{conv}}^{w^*}\lbrace \delta_x: x\in X\rbrace} \langle Q, f \rangle <r \leq \langle \mu, f \rangle.$$
On one hand, $\sup_{Q\in \overline{\textnormal{conv}}^{w^*}\lbrace \delta_x: x\in X\rbrace} \langle Q, f \rangle\geq \sup_{x\in X} f(x)$ and on the other hand, since $\mu$ is positive then $$\langle \mu, f \rangle\leq \langle \mu, \sup_{x\in X} f(x)\rangle=\left( \sup_{x\in X} f(x) \right ) \langle \mu, \boldsymbol{1}_X\rangle=\sup_{x\in X} f(x) .$$
Hence, a contradiction. Thus, $\lbrace \nu \in (\mathcal{C}^+_b(X))^* : \langle \nu, \boldsymbol{1}_X\rangle=1\rbrace\subset \overline{\textnormal{conv}}^{w^*}\lbrace \delta_x: x\in X\rbrace.$ The reverse inclusion is clear. 
\end{proof}

\begin{theorem} \label{theorem1} Let $X$ be a topological space and $A$ be a nonempty convex subset of $\mathcal{C}_b(X)$. Then,  either $A_1)$ or $A_2)$ is true, where:

$A_1)$ there exists $\Phi_0\in A$ such that $\underset{x\in X}{\sup} \Phi_0(x) < 0$, 

$A_2)$  there exists $\nu \in \overline{\textnormal{conv}}^{w^*}\lbrace \delta_x: x\in X\rbrace$ such that $\langle \nu, \Phi \rangle \geq 0$ for all $\Phi \in A$.

If moreover, we assume that $X$ is completely regular Hausdorff pseudocompact space and $A$ is uniformly bounded and relatively compact for the topology $\tau_p$ of pointwise convergence on $X$, then we obtain  in $A_2)$ that $\langle \nu, \Phi \rangle \geq 0$ for all $\Phi \in \overline{A}^{\tau_p}$.

\end{theorem}

\begin{proof} It is easy to see that for all $\mu \in \overline{\textnormal{conv}}^{w^*}\lbrace \delta_x: x\in X\rbrace$ and all $\Phi \in \mathcal{C}_b(X)$, we have $\langle \mu, \Phi \rangle\leq \sup_{x\in X} \Phi(x)$. So it is clear that $A_1)$ and $A_2)$ cannot be realised at the same time. However, we will show that one of the alternatives $A_1)$ or $A_2)$, is always realised.  Indeed, this follows from  the following alternatives respectively:   either $-\boldsymbol{1}_X \in \overline{\R^+ A+\mathcal{C}^+_b (X)}$ or $-\boldsymbol{1}_X \not\in \overline{\R^+ A+\mathcal{C}^+_b (X)}$, where the closure is taken in the Banach space $(\mathcal{C}_b(X), \|\cdot\|_{\infty})$.
\vskip5mm
{\bf Case 1.} Suppose  that  $-\boldsymbol{1}_X \in \overline{\R^+ A+\mathcal{C}^+_b(X)}$. In this case,  there are some $\lambda \geq 0$, $\Phi_0\in A$ and $h_0\in \mathcal{C}^+_b(X)$ such that 
\begin{eqnarray*}
\|h_0 + \lambda \Phi_0 +\boldsymbol{1}_X\|_{\infty}<\frac{1}{2}
\end{eqnarray*}
 It follows  that, for all $x \in X$,
\begin{eqnarray*}
 h_0(x)+ \lambda \Phi_0(x) + 1 < \frac{1}{2}.
\end{eqnarray*}
Notice from the above inequality and the positivity of $h_0$ that  necessarily we have $\lambda >0$ and so  we have $\Phi_0(x) < -\frac{1}{2 \lambda}$, for all $x\in X$.  We deduce that $\sup_{x\in X} \Phi_0 (x)\leq -\frac{1}{2\lambda }<0$. 
Thus, the alternative $A_1)$ is satisfied.
\vskip5mm
{\bf Case 2.} Suppose that  $-\boldsymbol{1}_X \not\in \overline{\R^+ A+\mathcal{C}^+_b(X)}$. Then, by the Hahn-Banach theorem, (using the fact that $\overline{\R^+ A+\mathcal{C}^+_b(X)}$ is a cone) there exists $\mu \in (\mathcal{C}_b(X))^* \setminus\lbrace 0\rbrace$ such that 
\begin{eqnarray} \label{sup}
 \langle \mu, h \rangle  \geq 0, \hspace{3mm} \forall h\in \R^+ A+\mathcal{C}^+_b (X).
\end{eqnarray}
From Proposition \ref{pcone}, $(\mathcal{C}^+_b(X))^*=\cup_{\lambda\geq 0} \lambda \overline{\textnormal{conv}}^{w^*}\lbrace \delta_x: x\in X\rbrace$. Since, $\mu\neq 0$ there exists $\lambda>0$ such that  $\mu=\lambda \nu$ and $\nu\in \overline{\textnormal{conv}}^{w^*}\lbrace \delta_x: x\in X\rbrace$. We see from $(\ref{sup})$, that for all $\Phi\in A$, $ \langle \nu, \Phi \rangle \geq 0$. Thus, the alternative $A_2)$ is given.

Now, if moreover we assume that $X$ is completely regular Hausdorff pseudocompact space,  then according to \cite[Theorem 2.6]{Wrf} (see also \cite{Pt, Ts}), the  pointwise topology and the weak topology $\sigma(C_b(X), (C_b(X))^*)$ agree on uniformly bounded, pointwise compact sets of continuous functions. Hence, if $A$ is uniformly bounded and relatively compact for the topology $\tau_p$ of pointwise convergence on $X$ then from the alternative $A_2)$, we obtain that for all $\Phi\in \overline{A}^{\tau_p}$, $\langle \nu, \Phi \rangle \geq 0.$
\end{proof}
\begin{Exemps} We can recover the result (mentioned above) of Fan, Glicksberg, and Hoffman in \cite{FGH} as follows: we apply  Theorem \ref{theorem1} with $X:=\lbrace f_i: i=1...n\rbrace$ equiped with the discrete distance and $A:=\textnormal{conv}\lbrace \chi_t: t\in C\rbrace\subset \mathcal{C}(X)\simeq \R^n$, where for each $t\in C$, the function $\chi_t : X\to \R$ is defined by $\chi_t(f)=f(t)$ for all $f\in X$. 
\end{Exemps}

\section{Minimax Theorems} \label{RS4}

This section is divided into two sub-sections. The first deals with minimax theorems in relation to Simons-like inequality. In the second, we will consider a minimax theorem in the framework of pseudocompact spaces. These theorems are consequences of the results in the previous section. 

\subsection{Minimax theorem and Simons-like inequality} 

The aim of this section is to establish Theorem \ref{km1} and Corollary \ref{km2}. We need the following lemma.

\begin{lemma} \label{corollary0} Let $X$ and $Y$ be nonempty sets and $f: X\times Y\to \R$ be a function such that $C_Y:=\lbrace f(\cdot, y): y\in Y\rbrace\subset \ell^{\infty}(X)$. Suppose that $f$ is $t$-convexlike on $X$ for some $t\in (0,1)$. Then, there exists a net $(x_\alpha)_{\alpha\in I}\subset X$ such that $$\sup_{y\in Y} \limsup_\alpha f(x_\alpha, y)\leq \sup_{\Psi\in \textnormal{conv}(C_Y)} \underset{x\in X}{\inf} \Psi(x).$$ 
If moreover, we assume that $f$ is supinf-concave on $Y$, then
$$\sup_{y\in Y} \limsup_\alpha f(x_\alpha, y)\leq \sup_{y\in Y} \underset{x\in X}{\inf} f(x,y).$$
\end{lemma}

\begin{proof} Set $r:=\sup_{\Psi\in \textnormal{conv}(C_Y)} \underset{x\in X}{\inf} \Psi(x)>-\infty$. The inequality is trivial if $r=+\infty$. Suppose that $r <+\infty$.  Let us equip $X$ with the discrete distance so that $\ell^{\infty}(X)=\mathcal{C}_b(X)$. By applying Theorem \ref{theorem1} with the convex set $A:=-\textnormal{conv}(C_Y)+r$ and knowing that the alternative $A_1)$ cannot be satisfied because of the choice of $r$, we get that there exists $\mu \in \overline{\textnormal{conv}}^{w^*}\lbrace \delta_x: x\in X\rbrace$ such that $\langle \mu, \Psi \rangle \leq r$ for all $\Psi \in C_Y$. There exists a net $(\mu_\alpha)_{\alpha\in I}\subset \textnormal{conv}(\delta(X))$ that  $w^*$-converges to $\mu$. For each $\alpha \in I$ there exists $n_\alpha\geq 1$, $x^\alpha_1,...,x^\alpha_{n_\alpha}\in X$ and $(\lambda^\alpha_1,...,\lambda^\alpha_{n_\alpha})\in \Delta_{n_\alpha}$ such that $\mu_\alpha=\sum_{i=1}^{n_\alpha} \lambda^\alpha_i \delta_{x^\alpha_i}$. Since $f$ is $t$-convexlike on $X$, using Fact 1 in $(\ref{Fa})$ we have: for every $\varepsilon >0$ and $\beta \in (0, \frac{1}{ n_\alpha})$ and for $(a^\alpha_1,...,a^\alpha_{n_\alpha})\in D_{n_\alpha}(t)$ such that $\max\lbrace |\lambda^\alpha_i -a^\alpha_i|: i=1,...,n_\alpha \rbrace <\varepsilon \beta$, there exists $x_\alpha\in X$ satisfying: $\forall y \in Y$, 
\begin{eqnarray*}
  f(x_\alpha ,y) &\leq& \sum_{i=1}^{n_\alpha} a^\alpha_i f(x^\alpha_i, y).
\end{eqnarray*}
Set $\nu_\alpha:=\sum_{i=1}^{n_\alpha} a^\alpha_i \delta_{x^\alpha_i}$. We see that $\|\nu_\alpha -\mu_\alpha\|=\|\sum_{i=1}^{n_\alpha} (a^\alpha_i -\lambda^\alpha_i) \delta_{x^\alpha_i}\|\leq \varepsilon \beta n_\alpha\leq \varepsilon$. It follows that for all $f(\cdot, y) \in C_Y$ and all $\alpha \in I$,
\begin{eqnarray*}
 f(x_\alpha ,y) &\leq& \sum_{i=1}^{n_\alpha} a^\alpha_i f(x^\alpha_i, y) \\
&=& \langle \mu_\alpha, f(\cdot, y)\rangle +\langle \mu_\alpha -\nu_\alpha , f(\cdot,y)\rangle\\
&\leq& \langle \mu_\alpha, f(\cdot, y) \rangle +\varepsilon \|f(\cdot,y)\|_{\infty}.
\end{eqnarray*}
So we have, for every $\varepsilon >0$, 
\begin{eqnarray*}
 \underset{\alpha}{\limsup} f(x_\alpha,y) &\leq&\underset{\alpha}{\limsup} \langle \mu_\alpha, f(\cdot,y)\rangle+\varepsilon\|f(\cdot,y)\|_{\infty} \\
&=&\langle \mu, f(\cdot,y)\rangle+\varepsilon\|f(\cdot,y)\|_{\infty}, \forall y\in Y.
\end{eqnarray*}
Passing to the limit when $\varepsilon \to 0$, we get that
\begin{eqnarray*}
 \underset{\alpha}{\limsup} f(x_\alpha,y) &\leq&\langle \mu, f(\cdot,y)\rangle\leq r:=\sup_{\Psi\in \textnormal{conv}(C_Y)} \underset{x\in X}{\inf} \Psi(x), \forall y\in Y.
\end{eqnarray*}
If moreover we assume that $f$ is supinf-concave on $Y$, then we have $$\sup_{\Psi\in \textnormal{conv}(C_Y)} \underset{x\in X}{\inf} \Psi(x)=\sup_{y\in Y} \underset{x\in X}{\inf} f(x,y).$$ This concludes the proof of the lemma.
\end{proof}
\begin{remark} \label{remsuite} The net  $(x_\alpha)_\alpha$ can be replaced by a sequence $(x_n)_n$ in Lemma \ref{corollary0}, whenever the space $\overline{\textnormal{span}}\lbrace f(\cdot,y): y \in Y\rbrace$ is assumed to be a  separable subspace of $\ell^{\infty}(X)$. Indeed, let us equip $X$ with the discrete distance so that $\mathcal{C}_b(X)$ coincides with  $\ell^{\infty}(X)$. The injective mapping $i: \overline{\textnormal{span}}\lbrace f(\cdot,y): y \in Y\rbrace \to \mathcal{C}_b(X)$ is an isometry and the adjoint $i^* : \mathcal{C}_b(X)^* \to (\overline{\textnormal{span}}\lbrace f(\cdot,y): y \in Y\rbrace)^*$ is surjective (see for instance \cite[Propoition 2]{Ba}) and $w^*$-to-$w^*$ continuous. Since $\overline{\textnormal{conv}}^{w^*}(\delta(X))$ is $w^*$-compact in $\mathcal{C}_b(X)^*$ then $K:=i^*(\overline{\textnormal{conv}}^{w^*}(\delta(X)))$ is $w^*$-compact in $(\overline{\textnormal{span}}\lbrace f(\cdot,y): y \in Y\rbrace)^*$. Since $\overline{\textnormal{span}}\lbrace f(\cdot,y): y \in Y\rbrace$ is separable, we have that $i^*(\overline{\textnormal{conv}}^{w^*}(\delta(X)))$ is metrizable. On the other hand, $\langle \mu, f(\cdot, y)\rangle=\langle i^*(\mu), f(\cdot, y)\rangle$ for all $\mu \in \overline{\textnormal{conv}}^{w^*}(\delta(X))$ and all $y\in Y$. Thus, the same proof as in Lemma \ref{corollary0} applies by changing the net $\mu_\alpha=\sum_{i=1}^{n_\alpha} \lambda^\alpha_i \delta_{x^\alpha_i}$, $\alpha \in I$ with a sequence $\mu_k=\sum_{i=1}^{n_k} \lambda^k_i \delta_{x^k_i}$, $k\in \N$.
\end{remark}

The following theorems are extensions of the known minimax theorems of Fan, K$\ddot{\textnormal{o}}$nig and Simons  in \cite[Theorem 11, Theorem 12 \& Theorem 26]{Simons1}. No topology on $X$ and $Y$ is assumed here.

\begin{theorem} \label{km1} Let $X$ and $Y$ be nonempty sets and $f, g: X\times Y\to \R$ be two functions. Suppose that: 

$(i)$ $f$ is bounded on $X$ and $t$-convexlike on $X$ for some $t\in (0,1)$,

$(ii)$ the Simons-like inequality holds, that is, for every net $(x_\alpha)_\alpha \subset X$,
$$\inf_{x\in X} \sup_{y\in Y} f(x,y)\leq \sup_{y\in Y} \limsup_\alpha f(x_\alpha, y).$$

$(iii)$ $g$ is supinf-concave on $Y$.

$(iv)$ $f\leq g$. 

Then, 
$$\inf_{x\in X} \sup_{y\in Y} f(x,y)\leq  \sup_{y\in Y} \inf_{x\in X} g(x, y).$$
\end{theorem}
\begin{proof}  Under the assumption $(i)$ and using Lemma \ref{corollary0}, there exists a net $(x_\alpha)_\alpha \subset X$ such that
$$ \sup_{y\in Y} \limsup_\alpha f(x_\alpha, y) \leq \sup_{\Psi\in \textnormal{conv}(C_Y)} \underset{x\in X}{\inf} \Psi(x).$$
Using the assumption $(ii)$, we obtain 
$$\inf_{x\in X} \sup_{y\in Y} f(x,y) \leq \sup_{\Psi\in \textnormal{conv}(C_Y)} \underset{x\in X}{\inf} \Psi(x).$$
Since, $f\leq g$, we see that $\sup_{\Psi\in \textnormal{conv}(C_Y)} \underset{x\in X}{\inf} \Psi(x)\leq \sup_{\Phi\in \textnormal{conv}(D_Y)} \underset{x\in X}{\inf} \Phi(x)$, where $D_Y:=\lbrace g(\cdot, y): y\in Y\rbrace$. Thus, using $(iii)$ we get $$\inf_{x\in X} \sup_{y\in Y} f(x,y)\leq  \sup_{y\in Y} \inf_{x\in X} g(x, y).$$ 
\end{proof}

The aim of Example \ref{Exp} and Example \ref{Exp1} below is to motivate the Simons-like inequality and minimax equality respectively, in a setting where neither $X$ nor $Y$ is assumed to be compact. The idea in Example \ref{Exp1} is to show that functions $f$ expressed as an inf-convolution of a rather general function $g$ and a specific function $\phi$ satisfy Simons' inequality and, therefore, also the minimax theorem in a non-compact setting.

\begin{Exemps} \label{Exp} Let $X$ and $Y$ be nonempty sets, $\phi: X\times X\to \R$ and $f: X\times Y\to \R$ be two functions satisfying:

$(a_1)$ For every net $(x_\alpha)_{\alpha \in I} \subset X$, there exists a subnet $(x_{h(x_\alpha)})$ such that $$\inf_{x\in X} \limsup_\alpha \phi(x,x_{h(\alpha)})\leq 0.$$

$(a_2)$ There exists $t\in (0,1)$ such that for every $x_1, x_2\in X$ there exists $x_3\in X$ such that $t\phi(x_3,x_1)+(1-t)\phi(x_3,x_2)\leq 0$.

$(a_3)$ For some $K\geq 0$,
$$f(x_1,y)-f(x_2,y)\leq K \phi(x_1,x_2), \forall x_1, x_2\in X, \forall y\in Y.$$

Then, $f$ is $t$-convexlike and satisfies the Simons-like inequality.
\vskip5mm
Notice that if $\xi:X\to \R$ is a function bounded from below and $\phi(x,z):=\max(\xi(x)-\xi(z),0)$, for all $x, z\in X$, then $\phi$ satisfies $(a_1)$-$(a_2)$. 
\end{Exemps}

\begin{proof} Let $I$ be a direct set and $(x_\alpha)_{\alpha\in I} \subset X$ be a net. From the assumptions there exists a subnet $(x_{h(\alpha)})_{\alpha\in I}$ such that  $\inf_{x\in X} \limsup_{\alpha} \phi(x,x_{h(\alpha)})=0$.  On the other hand, for all $\alpha \in I$, all $x\in X$ and all $y\in Y$,
\begin{eqnarray*}
f(x,y)\leq f(x_{h(\alpha)},y)+K \phi(x,x_{h(\alpha)})
\end{eqnarray*}
Then,  for all $x \in X$ and all $y\in Y$:
\begin{eqnarray*}
f(x,y)\leq \limsup_\alpha f(x_{h(\alpha)},y)+K \limsup_\alpha \phi(x, x_{h(\alpha)})
\end{eqnarray*}
So, we get 
\begin{eqnarray*}
\inf_{x\in X} \sup_{y\in Y} f(x,y) &\leq&  \sup_{y\in Y} \limsup_\alpha f(x_{h(\alpha)},y)+K\inf_{x\in X} \limsup_\alpha \phi(x, x_{h(\alpha)}) \\
&\leq&  \sup_{y\in Y} \limsup_\alpha f(x_{\alpha},y)
\end{eqnarray*}
Hence, the Simons-like inequality holds. Now, using $(a_2)$ and $(a_3)$, for every $x_1, x_2\in X$ there exists $x_3\in X$, such that: $\forall y\in Y$
\begin{eqnarray*} 
f(x_3,y) &\leq& t f(x_1,y)+(1-t)f(x_2,y)+ K[t\phi(x_3,x_1)+(1-t)\phi(x_3,x_2)]\\
&\leq& t f(x_1,y)+(1-t)f(x_2,y).
\end{eqnarray*}
Thus, $f$ is $t$-convexlike.
\end{proof}

In fact, for the one-function minimax theorem, it turns out that Simons-like inequality is a necessary and sufficient condition for obtaining minimax equality as soon as $f$ is bounded and $t$-convexlike on $X$ and $s$-concavelike on $Y$ for some $t, s\in (0,1)$.

\begin{corollary}  \label{km2} Let $X$ and $Y$ be nonempty sets and $f: X\times Y\to \R$ be a function. Suppose that $f$ is bounded on $X$, $t$-convexlike on $X$ for some $t\in (0,1)$ and supinf-concave on $Y$. 

Then, the following assertions are equivalent.

$(a)$ The Simons-like inequality holds, that is, for every net $(x_\alpha)_\alpha \subset X$,
$$\inf_{x\in X} \sup_{y\in Y} f(x,y)\leq \sup_{y\in Y} \limsup_\alpha f(x_\alpha, y).$$

$(b)$ The minimax equality holds:
$$\inf_{x\in X} \sup_{y\in Y} f(x,y)= \sup_{y\in Y} \inf_{x\in X} f(x, y).$$
\end{corollary}
\begin{proof} The implication $(a) \Longrightarrow (b)$ is a consequence of Theorem \ref{km1}. The implication $(b) \Longrightarrow (a)$ is always true and requires no assumptions about $X, Y$ and $f$. Indeed, clearly we have $\inf_{x\in X} f(x,y) \leq \underset{\alpha}{\limsup}f(x_{\alpha},y)$ for every net  $(x_\alpha)_\alpha \subset X$ and every $y\in Y$. Thus, $\inf_{x\in X} \sup_{y\in Y} f(x,y)=\sup_{y\in Y} \inf_{x\in X} f(x,y) \leq \sup_{y\in Y} \underset{\alpha}{\limsup}f(x_{\alpha},y)$. 
\end{proof}

We have the following general and generic example of minimax theorem which works for arbitrary sets $X$ and $Y$.

\begin{Exemps} \label{Exp1} Let $X$ and $Y$ be two nonempty sets, $\xi:X\to \R$ be a function bounded from below, $\phi(x,z)=\max(\xi(x)-\xi(z),0)$, for all $x, z\in X$ and $g: X\times Y\to \R$ be a function which is bounded on $X$ and $s$-concavelike on $Y$ for some $s\in (0,1)$. Let $K$ be any non-negative real number and define the function $f$ by 
$$f(x,y):=\inf_{z\in X}\lbrace g(z,y)+K\phi(x,z) \rbrace, \forall x\in X, \forall y\in Y,$$
Then, $$\inf_{x\in X} \sup_{y\in Y} f(x,y)= \sup_{y\in Y} \inf_{x\in X} f(x, y).$$
\end{Exemps}
\begin{proof} First, we see that $f$ is bounded on $X$. Indeed, since $\phi(x,x)=0$ for all $x\in X$, we have that $f(x,y)\leq g(x,y)$ for all $(x,y)\in X\times Y$. On the other hand, since $\phi\geq 0$, we have that $f(x,y)\geq \inf_{z\in X} g(z,y)$. Hence, $f$ is bounded on $X$, since $g$ is bounded on $X$. We see easily that $\phi(x,z)+\phi(z,\bar{x})\geq \phi(x,t)$ for all $x,\bar{x}, z\in X$, which makes it easy to prove that $f(x_1,y)-f(x_2,y)\leq K \phi(x_1,x_2), \forall x_1, x_2\in X, \forall y\in Y$. Indeed, for every $\varepsilon >0$, there exists $z_\varepsilon \in X$ such that 
$$f(x_2,y) +\varepsilon > g(z_\varepsilon,y)+K\phi(x_2,z_\varepsilon).$$
On the other hand,
$$f(x_1,y)\leq g(z_\varepsilon,y)+K\phi(x_1,z_\varepsilon).$$
Hence,
$$f(x_1,y)-f(x_2,y)\leq K[\phi(x_1,z_\varepsilon)-\phi(x_2,z_\varepsilon)] +\varepsilon\leq K\phi(x_1,x_2)+\varepsilon.$$
Passing to the limit when $\varepsilon \to 0$, we get that
$$f(x_1,y)-f(x_2,y)\leq K\phi(x_1,x_2), \forall x_1, x_2\in X, \forall y\in Y.$$
On the other hand, it is easy to see that $\phi$ satisfies $(a_1)$ and $(a_2)$ (for all $t\in (0,1)$) of Example \ref{Exp}.  Hence, by Example \ref{Exp}, $f$ is $t$-convexlike for every $t\in (0,1)$ and satisfies the Simons-like inequality. Finally, we see that $f$ is $s$-concavelike on $Y$ as the infimum of $s$-concavelike on $Y$ functions. We conclude using Corollary \ref{km2}.
\end{proof}

The Simons-like inequality is not always trivial and requires certain conditions, such as in Example \ref{Exp} or the one (harder) established in \cite{Simons3, DF, KO} (see the proof of Corollary \ref{corder} below). However, this inequality is easily and always satisfied if $X$ is a compact Hausdorff set and $f$ is lower semicontinuous on $X$. Thus, we recover easily in the following corollary the Simons minimax theorem (with a slight improvement, since here $g$ is assumed to be only supinf-concave on $Y$ instead of being $t$-concavelike for some $t\in (0,1)$).

\begin{corollary} $($\textnormal{Simons} \cite[Theorem 26]{Simons1} $\&$ \cite[Theorem 2]{Fk}$)$ \label{app2}  Let $Y$  be an arbitrary nonempty sets, $X$ is a compact Hausdorff  space and $f, g: X\times Y\to \R$ be two functions bounded on $X$. Assume that:

$(i)$ $f$ is $t$-convexlike for some $t\in (0,1)$ and lower semicontinuous on $X$,

$(ii)$ $g$ is supinf-concave on $Y$,

$(iii)$ $f\leq g$.

Then, 
$$\inf_{x\in X} \sup_{y\in Y} f(x,y) \leq \sup_{y\in Y}\inf_{x\in X} g(x,y).$$
\end{corollary}

\begin{proof}  Let $(x_\alpha)_{\alpha\in I}\subset X$ be a net. Then, there exists a subnet $(x_{h(\alpha)})_{\alpha\in I}$ that converges to some $\bar{x}\in X$ (since $X$ is compact). By the lower semicontinuity of $f$ on $X$, we have for all $y\in Y$, 
$f(\bar{x}, y)\leq \underset{\alpha}{\liminf} f(x_{h(\alpha)},y) \leq \underset{\alpha}{\limsup} f(x_\alpha,y)$. Thus, $\underset{y\in Y}{\sup} f(\bar{x},y)\leq \underset{y\in Y}{\sup} \underset{\alpha}{\limsup} f(x_\alpha,y)$. Finally, we have $$\underset{x\in X}{\inf} \underset{y\in Y}{\sup} f(x,y)\leq \underset{y\in Y}{\sup} \underset{\alpha}{\limsup} f(x_\alpha,y).$$ Hence, all of the assumptions of Theorem ~\ref{km1} are satisfied, which gives the conclusion. 
\end{proof}

\begin{remark} The necessary conditions to obtain a general two-function minimax results are given in \cite[Lemma 3.5]{RG1}. We can compare the fairly general results of \cite[Theorem 3.8 \& Theorem 3.9]{RG1} with the result in the above corollary.
\end{remark}

Thanks to Remak ~\ref{remsuite}, if we assume that $\overline{\textnormal{span}}\lbrace f(\cdot,y): y \in Y\rbrace$ is a separable subspace of $\ell^{\infty}(X)$, then we can replace the nets $(x_\alpha)_\alpha \subset X$ by sequances $(x_n)_n\subset X$ in the results of this section. Using Theorem ~\ref{km1} and  \cite[Theorem 2]{KO} (see also \cite[Theorem 1]{DF}) we obtain the following minimax result. No assumption of compactness or semicontinuity is required. 

\begin{corollary} \label{corder} Let $Y$ be a nonempty set and let $X$ be a subset of a Hausdorﬀ topological vector space which is invariant under inﬁnite convex combinations. Let $H, L : X \times Y \to \R$ be two mappings such that 

$(i)$ $H$ is bounded on $X\times Y$ and the functions $H(\cdot , y) : X\to \R$, 
$y \in Y$, are convex and $H(\lambda x, y)=\lambda H(x,y)$, whenever $\lambda >0$, $\lambda x\in X$ and $y\in Y$. Assume moreover that, for every $x\in X$, there exists $y_x \in Y$ satisfying
$$H(x, y_x) = \sup_{y\in Y} H(x, y).$$

$(ii)$ $\overline{\textnormal{span}}\lbrace H(\cdot,y): y \in Y\rbrace$ is a separable subspace of $\ell^{\infty}(X)$. 

$(iii)$ $L$ is supinf-concave on $Y$,

$(iv)$ $H\leq L$.

Then, 
$$\inf_{x\in X} \sup_{y\in Y} H(x,y) \leq \sup_{y\in Y}\inf_{x\in X} L(x,y).$$
\end{corollary}

\begin{proof} Under the assuption $(i)$, using  \cite[Theorem 2]{KO}, we get that for every sequence $(x_n)_n \subset X$, $$\underset{x\in X}{\inf}\underset{y\in Y}{\sup} H(x,y)\leq \underset{y\in Y}{\sup } \limsup_n H(x_n,y). 
$$
We conclude using Remak ~\ref{remsuite}  and Theroem ~\ref{km1}.

\end{proof}
We end this section with a minimax equality in Proposition \ref{SI} below, based on a previous work in \cite{Ba1}. Let $X$ be a topological space and $(Y,\|\cdot\|)$ be a Banach space included in $\mathcal{C}_b(X)$ and such that $\|\cdot\|\geq \|\cdot\|_{\infty}$ on $Y$.  For a bounded from below function $h: X\to \R$, we define the conjugate $h^\times: Y\to \R$ by 
$$h^\times (\xi):=\sup_{x\in X}\lbrace \xi(x)-h(x) \rbrace, \forall \xi \in Y.$$
A non-convex duality analogous to Fenchel duality based on this conjunction was introduced in \cite{Ba2} and studied in \cite{Ba2, Ba1}. Let us recall from \cite{Ba1} the following notions. 

 A subset $A$ of $Y$ is said to be a $\Delta_Y$-set, if and only if there exists real numbers $(\lambda_x)_{x\in X}\in \R^X$ such that 
$$A=\cap_{x\in X}\lbrace \phi\in Y: \phi(x)\leq \lambda_x\rbrace .$$
Clearly, a $\Delta_Y$-set is a convex $\tau_p$-closed subset of $(Y,\tau_p)$.  We say that the pair $(X,Y)$ satisfies the property $(H)$ if and only if, for each $x\in X$ and each open neighborhood $U$ of $x$, there exists $\sigma  : X\longrightarrow [0,1]$ such that $\sigma \in Y$, $\sigma (x)=1$ and $\sigma (y)=0$ for all $y \in X\setminus U$. Tanks to the Urysohn's lemma, if $X$ is a normal Haudroff space and $Y=\mathcal{C}_b(X)$ then $(X,Y)$ has the property $(H)$. Other examples can be found in \cite[Examples ~1]{Ba1}.

\begin{proposition}\label{SI} Let $X$ be a normal Hausdorff space, $(Y,\|\cdot\|)$ be a Banach space included in $\mathcal{C}_b(X)$ such that $\|\cdot\|\geq \|\cdot\|_{\infty}$ and $(X,Y)$ has the property $(H)$. Let $A\subset Y$ be a  $\Delta_Y$-set.
 Then, $$\inf_{x\in X}\sup_{\phi\in A} \phi(x)=\sup_{\phi\in A} \inf_{x\in X}\phi(x).$$
\end{proposition}
\begin{proof} First, we prove that there exists a bounded from below lower semicontinuous function $f: X\to \R$ such that $A=A_{Y}(f):=\lbrace \psi \in Y: \psi \leq f\rbrace$. Indeed, since $A$ is a $\Delta_Y$-set, there exists real numbers $\lambda_x \in \R$, for all $x\in X$, such that $A=\cap_{x\in X}\lbrace \varphi \in Y: \varphi(x)\leq \lambda_x \rbrace$. Let us set $f(x):= \sup_{\psi \in A } \psi(x)$, for all $x\in X$. Thus, we have $f(x)\leq \lambda_x< +\infty$ for all $x\in X$. It follows that $f$ is lower semicontinuous as supremum of continuous function and $A_Y(f)\subset A$. Moreover, $f$ is bounded from below, since there exists a bounded continuous function $\varphi \in A\subset \mathcal{C}_b(X)$ such that $-\infty < \inf_X \varphi \leq \varphi \leq f$. On the other hand, if $\varphi\in A$, then for all $x\in X$ we have $\varphi(x) \leq \sup_{\psi \in A } \psi(x):=f(x)$. This shows that $\varphi\in A_Y(f)$ and so that $A\subset A_Y(f)$. Hence $A=A_Y(f)$.  By \cite[Lemma 3]{Ba1} we have $f^\times (0)=\inf_{\phi \in A} \phi^\times(0)$, which is equivalent by the definitions to $\inf_{x\in X} f(x)=\sup_{\phi \in A}\inf_{x\in X} \phi(x)$. Finally, we obtain that $\inf_{x\in X}\sup_{\phi\in A} \phi(x)=\sup_{\phi\in A} \inf_{x\in X}\phi(x).$
\end{proof}

\subsection{Minimax theorem in pseudocompact space}

The aim of this section is to establish Theorem \ref{cor-equic}. We need the following two lemmas.
\begin{lemma} \label{lem1} Let $\Omega$ be a completely regular Hausdorff  pseudocompact space and $\Phi$ be a nonempty  subset of $\mathcal{C}(\Omega)$ such that $\textnormal{conv} (\Phi)$ is uniformly bounded and relatively compact for the pointwize topology (in particular if $\Phi$ is uniformly bounded and equi-continuous). Then, 
$$\inf_{\phi\in \overline{\textnormal{conv}}^{\tau_p} (\Phi)} \sup_{u\in \Omega} \phi(u)=\inf_{\phi\in \textnormal{conv} (\Phi)} \sup_{u\in \Omega} \phi(u),$$
where $\tau_p$ denotes the topology of pointwize convergence.
\end{lemma} 

\begin{proof}  Suppose by contradiction that there exists $r\in \R$ such that 
\begin{eqnarray} 
\label{contr} \inf_{\phi\in \overline{\textnormal{conv}}^{\tau_p} (\Phi)} \sup_{u\in \Omega} \phi(u) < r \leq \inf_{\phi\in \textnormal{conv} (\Phi)} \sup_{u\in \Omega} \phi(u).
\end{eqnarray}
By assumption $\Omega$ is pseudocompact and the set $\textnormal{conv} (\Phi) -r$ is relatively compact  in $\mathcal{C}(\Omega)$ for the pointwize convergence on $\Omega$. Thus, by applying Theorem ~\ref{theorem1} (the second part of the theorem) with the pseudocompact  $\Omega$ and the set $A:=\textnormal{conv}(\Phi) - r\subset \mathcal{C}(\Omega)$, we get the following alternatives: either $A_1)$ or $A_2)$ is true, where 

$A_1)$ there exists $\bar{\phi}\in\textnormal{conv}(\Phi)$ such   
\begin{eqnarray} \label{Q1}
\sup_{u\in \Omega}  \bar{\phi}(u) < r.
\end{eqnarray}

$A_2)$ there exists $\nu^* \in \overline{\textnormal{conv}}^{w^*}\lbrace \delta_u: u\in \Omega\rbrace \subset \mathcal{C}(\Omega)^* \setminus \lbrace 0\rbrace$ such that $\langle \nu^*, \phi -r \rangle \geq 0 $ for all $\phi \in \overline{\textnormal{conv}}^{\tau_p} (\Phi)$. In this case, let $(\nu^*_\alpha)_\alpha\subset \textnormal{conv} (\lbrace\delta_u: u\in \Omega\rbrace)$ be a net $w^*$-converging to $\nu^*$. We see that $\langle \nu^*_\alpha, \phi \rangle \leq \sup_{u\in \Omega} \phi(u)$ for all $\phi\in \overline{\textnormal{conv}}^{\tau_p} (\Phi)$. Thus, taking the $w^*$-limit we get $r\leq \langle \nu^*, \phi \rangle \leq \sup_{u\in \Omega}  \phi(u)$ for all $\phi\in \overline{\textnormal{conv}}^{\tau_p} (\Phi)$. Hence, we have
\begin{eqnarray} \label{Q2}
\inf_{\phi\in \overline{\textnormal{conv}}^{\tau_p} (\Phi)} \sup_{u\in \Omega} \phi(u) \geq r.
\end{eqnarray}
Both the formulas in $(\ref{Q1})$ and $(\ref{Q2})$ contradict $(\ref{contr})$.
\end{proof}

\begin{lemma} \label{lem2} Let $X$ and $Y$ be arbitrary nonempty sets and  $f: X\times Y\to \R$ be a function. Then, the following assertions hold.

$(i)$ If $f$ is infsup-convex on $X$ and bounded on $Y$, then there exists $\nu \in \overline{\textnormal{conv}}^{w^*}(\delta(Y))$ such that
\begin{eqnarray*} \label{Eq1l2}
\inf_{x\in X}\sup_{y\in y} f(x,y)= \inf_{x\in X} \underset{\mu\in \overline{\textnormal{conv}}^{w^*}(\delta(Y))}{\sup} \langle \mu, f(x,\cdot)\rangle &=& \underset{\mu\in \overline{\textnormal{conv}}^{w^*}(\delta(Y))}{\sup} \inf_{x\in X}\langle \mu, f(x,\cdot)\rangle\\
& =& \inf_{x\in X}\langle \nu, f(x,\cdot)\rangle.
\end{eqnarray*}

$(ii)$ If $f$ is supinf-concave on $Y$ and bounded on $X$, then there exists $\nu \in \overline{\textnormal{conv}}^{w^*}(\delta(X))$ such that
\begin{eqnarray*} 
\sup_{y\in Y}\inf_{x\in X} f(x,y)=\sup_{y\in Y} \underset{\mu\in \overline{\textnormal{conv}}^{w^*}(\delta(X))}{\inf} \langle \mu, f(\cdot,y)\rangle &=& \underset{\mu\in \overline{\textnormal{conv}}^{w^*}(\delta(X))}{\inf} \sup_{y\in Y}\langle \mu, f(\cdot,y)\rangle\\
& =& \sup_{y\in Y}\langle \nu, f(\cdot, y)\rangle.
\end{eqnarray*}

\end{lemma}

\begin{proof} $(i)$ We equip $Y$ with the discrete distance so that $\ell^{\infty}(Y)=\mathcal{C}_b(Y)$. From the boundedness of $f$ on $Y$, we see that $f(x,\cdot)\in \mathcal{C}_b(Y)$, for all $x\in X$. Notice that we always have
\begin{eqnarray} \label{coucou} 
\underset{\mu\in \overline{\textnormal{conv}}^{w^*}(\delta(Y))}{\sup} \inf_{x\in X}\langle \mu, f(x,\cdot)\rangle &\leq&  \inf_{x\in X} \underset{\mu\in \overline{\textnormal{conv}}^{w^*}(\delta(Y))}{\sup} \langle \mu, f(x,\cdot)\rangle \nonumber \\
&=& \inf_{x\in X} \underset{\mu\in \delta(Y)}{\sup} \langle \mu, f(x,\cdot)\rangle \nonumber \\
&=& \inf_{x\in X} \underset{y\in Y}{\sup} f(x,y).
\end{eqnarray}
We see that $\inf_{x\in X}\sup_{y\in Y} f(x,y)<+\infty$, by the boundedness of $f$ on $Y$. If $\inf_{x\in X}\sup_{y\in Y} f(x,y)=-\infty$, then the desired formula is trivial. Assume that  $$r:=\inf_{x\in  X}\sup_{y\in Y} f(x,y)\in \R.$$ 
Using Theorem ~\ref{theorem1} with  the convex set $A:=\textnormal{conv}\lbrace f(x,\cdot) -r: x \in X\rbrace \subset \mathcal{C}_b(Y)$, we obtain the following alternatives: either $A_1)$ or $A_2)$ is true, where

$A_1)$ there exists $x_1,...,x_n \in X$ and $\lambda_1,...\lambda_n\geq 0$ such that $\sum_{i=1}^n \lambda_i=1$ and 
$$\sup_{y\in Y} \sum_{i=1}^n \lambda_i (f(x_i,y) -r)<0,$$
equivalently,
$$\sup_{y\in Y} \sum_{i=1}^n \lambda_i f(x_i,y) <r.$$

$A_2)$ there exists $\nu \in \overline{\textnormal{conv}}^{w^*}\lbrace \delta_y: y\in Y\rbrace $, such that $\langle \nu, \Phi\rangle\geq 0$ for all $\Phi\in A$. In particular, we have $\inf_{x\in X}\langle \nu, f(x,\cdot)\rangle\geq r$.

\vskip5mm
The alternative $A_1)$ is not satisfied. Indeed, suppose by contradiction that $A_1)$ is true. Then, by setting $C_X:=\lbrace f(x, \cdot): x\in X\rbrace$ and by using the infsup-convexity of $f$ on $X$, we get, 
\begin{eqnarray*} 
r:=\inf_{x\in X} \sup_{y\in Y} f(x,y) &=& \inf_{\Psi\in \textnormal{conv}(C_X)} \underset{y\in Y}{\sup} \Psi(y) \\
&\leq&  \sup_{y\in Y} \sum_{i=1}^n \lambda_i  f(x_i,y) \\
& <& r \textnormal{ (by the alternative } A_1).
\end{eqnarray*}
A contradiction.  Hence, the alternative $A_2)$ is true and so the part $(i)$ is obtained by combining the alternative $A_2)$ and the inequality in $(\ref{coucou})$.

$(ii)$ This part is similar to $(i)$ it suffices to apply $(i)$ with the function $\widetilde{f}$ defined from $Y\times X$ into $\R$ by $\widetilde{f}(y,x)= -f(x,y)$, for all $(y,x)\in Y\times X$.

\end{proof}

Recall that one form of the Arzel\`a-Ascoli theorem is the following:  Let $Z$ be a completely regular Hausdorff space and $\Phi \subset \mathcal{C}_b(Z)$ be a uniformly bounded and equicontinuous set of continuous functions, then $\Phi$  is relatively compact in $\mathcal{C}(Z)$ for the topology of  pointwize convergence on $Z$ (see \cite{Wrf}). Notice also that if $\Phi$ is a uniformly bounded and  equicontinuous set then $\textnormal{conv} (\Phi)$ is also a uniformly bounded and equicontinuous set.

In the following result, we will not assume the $t$-convexlikness but only the infsup-convexity. 

\begin{theorem} \label{cor-equic} Let $X$ be an arbitrary  nonempty set and $Y$ be a completely regular Hausdorff pseudocompact space. Let $ H,L: X\times Y \to \R$ be two mappings, with $H$ bounded on $X\times Y$ and $L$ bounded on $X$ and satisfying: 

$(i)$ $H$ is infsup-convex on $X$ and the family $\textnormal{conv}\left (\lbrace H(x,\cdot): x\in X\rbrace \right )$ is relatively compact in $\mathcal{C}(Y)=\mathcal{C}_b(Y)$, for the pointwize topology on $Y$ (in particular if $\lbrace H(x,\cdot): x\in X\rbrace$  is equi-continuous  on $Y$),

$(ii)$ $L$ is supinf-concave on $Y$,

$(iii)$ $H\leq L$.

Then,

$$\inf_{x\in X} \sup_{y\in Y} H(x,y) \leq \sup_{y\in Y}\inf_{x\in X} L(x,y).$$

\end{theorem}
\begin{proof} Since, $H$ is bounded on $X\times Y$, then $\Phi:=\lbrace \delta_x\circ H:=H(x,\cdot): x\in X\rbrace$ is uniformly bounded in $\mathcal{C}_b(Y)$. We see by the definitions that $$\overline{\textnormal{conv}}^{\tau_p} (\Phi)=\lbrace  \mu \circ H: \mu \in  \overline{\mathrm{conv}}^{w^*}\left(\delta(X)\right)\rbrace,$$
$$\textnormal{conv} (\Phi)=\lbrace \mu \circ H: \mu \in  \mathrm{conv}\left(\delta(X)\right)\rbrace.$$
Thus, using Lemma \ref {lem1} we get
$$\inf_{\mu \in  \overline{\mathrm{conv}}^{w^*}\left(\delta(X)\right)} \sup_{y\in Y} \langle \mu, H(\cdot,y)\rangle=\inf_{\mu \in  \mathrm{conv}\left(\delta(X)\right)} \sup_{y\in Y} \langle \mu, H(\cdot,y)\rangle.$$
Since, $H$ is infsup-convex on $X$ then, 
$$\inf_{x\in X} \sup_{y\in Y} H(x,y)=\inf_{\mu \in  \mathrm{conv}\left(\delta(X)\right)} \sup_{y\in Y} \langle \mu, H(\cdot,y)\rangle =\inf_{\mu \in  \overline{\mathrm{conv}}^{w^*}\left(\delta(X)\right)} \sup_{y\in Y} \langle \mu, H(\cdot,y)\rangle.$$
Since $H\leq L$  and $\mu$ is positive whenever $\mu\in \overline{\textnormal{conv}}^{w^*}\left(\delta(X)\right)$, we have $\langle \mu, H(\cdot,y)\rangle \leq \langle \mu, L(\cdot,y)\rangle$, for all $y \in Y$ and all $\mu\in \overline{\textnormal{conv}}^{w^*}\left(\delta(X)\right)$. Using this fact and the above formulas, it follows 
\begin{eqnarray*}
 \underset{x\in X}{\inf} \underset{y\in Y}{\sup} H(x,y)&=&\underset{\mu\in \overline{\textnormal{conv}}^{w^*}\left(\delta(X)\right)}{\inf} \underset{y\in Y}{\sup} \langle \mu, H(\cdot,y)\rangle\\
&\leq& \underset{\mu\in \overline{\textnormal{conv}}^{w^*}\left(\delta(X)\right)}{\inf} \underset{y\in Y}{\sup} \langle \mu, L(\cdot,y)\rangle.
\end{eqnarray*}
Using the part $(ii)$ of Lemma ~\ref{lem2}, since $L$ is supinf-concave on $Y$ and bounded on $X$,
\begin{eqnarray*}
\sup_{y\in Y}\inf_{x\in X} L(x,y)=  \underset{\mu\in \overline{\textnormal{conv}}^{w^*}\left(\delta(X)\right)}{\inf} \underset{y\in Y}{\sup} \langle \mu, L(\cdot,y)\rangle.
\end{eqnarray*}
Hence, 
\begin{eqnarray*}
 \underset{x\in X}{\inf} \underset{y\in Y}{\sup} H(x,y) \leq \sup_{y\in Y}\inf_{x\in X} L(x,y).
\end{eqnarray*}
\end{proof}

\section{Applications to $w^*$-pointwise convergent sequence.} \label{RS5}
As an application, we obtain in Proposition \ref{leto} and Corollary \ref{let} extensions of \cite[Proposition, p. 103]{DF} and \cite[Corollary 2]{KO}. In Proposition \ref{leto}, we deal with the real-valued case replacing compactness by pseudocompactness and in Corollary \ref{let}, we deal with the vector-valued case replacing  compactness by countably compactness and the  pointwize convergence will be replaced by the more general notion of $w^*$-pointwize convergence. 

\begin{proposition} \label{leto}  Let $Z$ be a completely regular Hausdorff pseudocompact space and $(f_n)_n$ be a uniformly bounded sequence of continuous functions from $Z$ into $\R$. Suppose that  

$(i)$ the sequence $(f_n)_n$ pointwize converges to $0$ on $Z$, 

$(ii)$ the set $\textnormal{conv}\lbrace  f_n:  n\in \N \rbrace$ is relatively compact in $\mathcal{C}(Z)$ for the pointwize topology on $Z$. 

Then, there exists a sequence of linear convex combinations of $(f_n)$ which is uniformly convergent to $0$ on $Z$ (that is $\|f_n\|_{\infty} \to 0$).

\end{proposition}
\begin{proof} Define $C=\textnormal{conv}\lbrace f_n:n\in \N \rbrace$ and the function $F: C\times Z\times [-1,1]\to \R$ by $F(f,x,t)=tf(x)$ and $\Phi:=\lbrace F(f, \cdot, \cdot): f\in C \rbrace \subset \mathcal{C}(Z\times [-1,1])$. We are going to apply Theorem \ref{cor-equic} with the functions $H=L=F$, the convex set $X=C$ and the completely regular Hausdorff pseudocompact space $Y=Z\times [-1,1]$ (see \cite[Lemma 4.3]{Strm}). For this, we need to show that $F$ is infsup-convex on $C$ and supinf-concave on $Z\times [-1,1]$ and that the $\textnormal{conv}(\Phi)$ is uniformly bounded and relatively compact for the pointwize topology in $\mathcal{C}(Z\times [-1,1])$. Indeed, 

$(a)$ $F$ is convex on $C$, hence it is in particular infsup-convex on $C$,

$(b)$ $\textnormal{conv}(\Phi)=\Phi$ is uniformly bounded in $\mathcal{C}(Z\times [-1,1])$ since $(f_n)_n$ is uniformly bounded. On the other hand, using the assumption $(ii)$, it is clear that $\textnormal{conv}(\Phi)$ is relatively compact for the pointwize topology in $\mathcal{C}(Z\times [-1,1])$. 

$(c)$ $F$ is supinf-concave on $Z\times [-1,1]$. Indeed,  let  $n\geq 1, x_{1},...,x_{n} \in Z, t_{1},...,t_{n} \in [-1,1],  (a_{1},...,a_{n})\in \Delta_n$, we have  
\begin{eqnarray*}
\inf_{f\in C}\sum_{i=1}^n a_{i} t_{i}f(x_{i})&\leq& \sum_{i=1}^n a_{i} t_{i}f_m(x_{i}), \forall m\in \N.
\end{eqnarray*}
By taking the limit when $m\to+\infty$ in the above inequality and using $(i)$, we see that: $\forall n\geq 1, \forall x_{1},...,x_{n} \in Z, \forall t_{1},...,t_{n} \in [-1,1], \forall (a_{1},...,a_{n})\in \Delta_n$
\begin{eqnarray*}
\inf_{f\in C}\sum_{i=1}^n a_{i} t_{i}f(x_{i}) &\leq& 0.
\end{eqnarray*}
Hence, 
\begin{eqnarray*} 
 \sup_{\underset{\underset{(\lambda_1,\lambda_2,...,\lambda_n)\in \Delta_n}{(x_1,t_1), ...,(x_n,t_n)\in Z\times [-1,1]}}{n\geq 1}} \inf_{x\in X}\sum_{i=1}^n a_i F(f,x_i, t_i) \leq 0
\end{eqnarray*}
On the other hand, 
$$\sup_{(x,t)\in Z\times [-1,1]} \inf_{f\in C} F(f, x,t) \geq \sup_{x\in Z} \inf_{f\in C} F(f, x,0)  =0.$$
Thus, 
$$\sup_{(x,t)\in Z\times [-1,1]} \inf_{f\in C}F(f, x,t) =  \sup_{\underset{\underset{(\lambda_1,\lambda_2,...,\lambda_n)\in \Delta_n}{(x_1,t_1), ...,(x_n,t_n)\in Z\times [-1,1]}}{n\geq 1}} \inf_{x\in X}\sum_{i=1}^n a_i F(f,x_i, t_i)= 0.$$
This shows that $F$ is supinf-concave on $Z\times [-1,1]$.

Now, from $(a)$, $(b)$ and $(c)$ and Theorem \ref{cor-equic}, we get that $$\inf_{f\in C} \|f\|_{\infty} = \inf_{f\in C} \sup_{(x,t)\in Z\times [-1,1]} F(f,x,t) = \sup_{(x,t)\in Z\times [-1,1]} \inf_{f\in C} F(f,x,p) =0.$$
Hence, there exists a sequence of linear convex combinations of $(f_n)$ which is uniformly convergent to $0$ on $Z$.
\end{proof}

We need the following lemma.
Recall that in a non-completely metrizable locally convex space, the closed convex hull of a compact set is not compact in general (see \cite[Example 3.34, p. 185]{AB}). This explains why we assumed in Lemma \ref{lem1} that $\textnormal{conv}(\Phi)$ is relatively compact in $\mathcal{C}(Z)$ for the pointwize topology on $Z$. However, as it is shown in the following lemma, when $Z$ is a countably compact space  and $\Phi$ is nonempty uniformly bounded and relatively compact subset in $\mathcal{C}(Z)$, then the $\tau_p$-closed convex hull $\overline{\textnormal{conv}}^{\tau_p}(\Phi)$ is also compact in $\mathcal{C}(Z)$ for the pointwize topology.

\begin{lemma} \label{relative} Let $Z$ be a countably compact space and $\Phi$ be a nonempty uniformly bounded and relatively compact subset in $\mathcal{C}(Z)$ for the pointwize topology denoted $\tau_p$. Then, the $\tau_p$-closed convex hull $\overline{\textnormal{conv}}^{\tau_p}(\Phi)$ is compact in $\mathcal{C}(Z)$ for the pointwize topology.
\end{lemma}

\begin{proof} Since  $\overline{\textnormal{conv}}^{\tau_p}(\Phi)$ is uniformly bounded and $\tau_p$-closed it is $\tau_p$-compact in $\ell^{\infty}(Z)$. Indeed, since $\ell^{\infty}(Z)=(\ell^1(Z))^*$, then  the topology $\tau_p$ and the weak$^*$ topology coincide on the bounded set $\overline{\textnormal{conv}}^{\tau_p}(\Phi)$ (see \cite[Proposition 2.34]{We}) and so we apply the Banach-Alaoglu-Bourbaki theorem. To conclude, it remains to show  that    $\overline{\textnormal{conv}}^{\tau_p}(\Phi) \subset \mathcal{C}(Z)$. Set $K:=(\overline{\textnormal{conv}}^{\tau_p}(\Phi),\tau_p)$ a compact Hausdorff  space. For each $x\in Z$, the evaluation mapping  $\chi_x : \phi\mapsto \phi(x)$ is bounded and continuous on $K$, hence $\chi_x \in \mathcal{C}(K)$. Moreover, the family $(\chi_x)_{x\in Z}$ is uniformly bounded on $K$ since $K$ is uniformly bounded on $Z$. Let $(\psi_\alpha)_\alpha\subset \textnormal{conv}(\Phi)$ be a net. Since $\overline{\textnormal{conv}}^{w^*}\lbrace \delta_\phi: \phi\in K \rbrace$ is $w^*$-compact in the dual space $(C(K))^*$, there exists a subnet $(\delta_{\psi_{h(\alpha)}})_\alpha$ that $w^*$-converges to some $\mu^*$ in $(\mathcal{C}(K))^*$. In particular, $\delta_{\psi_{h(\alpha)}}(\chi_x):=\psi_{h(\alpha)}(x)$ converges to $\mu^*(\chi_x)$ in $\R$ for every $x\in Z$. That is, $(\psi_{h(\alpha)})_\alpha$ converges pointwize to $\mu^*\circ \chi$. It remains to show that $\mu^*\circ \chi \in \mathcal{C}(Z)$, in other words, that $\mu^*\circ \chi: Z \to \R$ is continuous. Indeed, clearly the map $\chi: Z \to (\mathcal{C}(K), \tau_p)$ defined by $\chi(x):=\chi_x$ is continuous, so using \cite[Corollary 1.9 (a)]{Re}, the set $\chi(Z)$ is relatively compact in $(\mathcal{C}(K), \tau_p)$ as image by the continuous mapping $\chi$ of the countably compact set $Z$. According to Grothendieck's theorem \cite[Theorem 5]{Ga}, the  pointwise topology and the weak topology $\sigma(C(K), (C(K))^*)$ agree on $\chi(Z)$.  It follows that $\mu^* \circ \chi$ is continuous on $Z$. Finally, we proved that every net $(\psi_\alpha)_\alpha\subset \textnormal{conv}(\Phi)$ has a subnet converging pointwize to an element of $\mathcal{C}(Z)$, that is, $\overline{\textnormal{conv}}^{\tau_p}(\Phi)$ is $\tau_p$-compact in $\mathcal{C}(Z)$.
\end{proof}

 If we assume that $Z$ is a countably compact space (a particular case of pseudocompact spaces \cite{Strm}), then the condition $(ii)$ in Proposition \ref{leto} is not a necessary condition because it would follow from $(i)$ thanks to Lemma \ref{relative}. Thus, we obtain the following corollary.

Let $Z$ be a topological space, $E$ be a Banach space and $(f_n)_n$ be a sequence of continuous functions from $Z$ into $E$. We say that the $E$-valued sequence $(f_n)_n$, $w^*$-pointwize converges to $0$ on $Z$ if and only if the real-valued sequence $(p\circ f_n)_n$ pointwize converges to $0$ on $Z$ for every $p\in B_{E^*}$ (the closed unit ball of the dual $E^*$). Clearly, $(f_n)_n$, $w^*$-pointwize converges to $0$ on $Z$ whenever $(f_n)_n$, pointwize converges to $0$ on $Z$ (that is $\|f_n(x)\|_E\to 0$, $\forall x\in Z$).

\begin{corollary} \label{let}  Let $Z$ be a completely regular Hausdorff countably compact space, $E$ be a Banach space and $(f_n)_n$ be a uniformly bounded sequence of continuous functions from $Z$ into $E$. Suppose that  the sequence $(f_n)_n$, $w^*$-pointwize converges to $0$ on $Z$. Then, there exists a sequence of linear convex combinations of $(f_n)$ which is uniformly convergent to $0$ on $Z$ (that is $\|f_n\|_{\infty} \to 0$).
\end{corollary}
\begin{proof}  Define $\widetilde{f}_n: Z\times B_{E^*}\to \R$, by $\widetilde{f}_n(x,p)=p\circ f_n(x)$ for all $(x, p)\in Z\times B_{E^*}$ and for all $n\in \N$. Using the assumptions, we see easily that the real-valued sequence $(\widetilde{f}_n)_n$ is uniformly bounded sequence of continuous functions from $Z\times B_{E^*}$ into $\R$ and that  the sequence $(\widetilde{f}_n)_n$, pointwize converges to $0$ on $Z\times B_{E^*}$. Since $Z$ is Hausdorff countably compact space and $B_{E^*}$ is $w^*$-compact (by the Banach-Alaoglu-Bourbaki theorem), it follows that the product space $Z\times B_{E^*}$ is completely regular Hausdorff countably compact space for the product topology (see \cite{Rt}). Using Lemma \ref{relative}, we get that the set $\textnormal{conv}\lbrace  \widetilde{f}_n:  n\in \N \rbrace$ is relatively compact in $\mathcal{C}(Z\times B_{E^*})$ for the pointwize topology on $Z\times B_{E^*}$. Hence, from Proposition \ref{leto}, there exists a sequence of linear convex combinations of $(\widetilde{f}_n)$ which is uniformly convergent to $0$ on $Z\times B_{E^*}$. Equivalently, there exists a sequence of linear convex combinations of $(f_n)$ which is uniformly convergent to $0$ on $Z$.
\end{proof}


\bibliographystyle{amsplain}

\begin{thebibliography}{999}
%
\bibitem{AB} C. D. Aliprantis, K. C. Border, {\it Inﬁnite dimensional analysis} A hitchhiker’s guide (3rd Edition), Springer (2006). 
%
\bibitem{Ba} M. Bachir, {\it Porosity in the space of Holder-functions}, J. Math. Anal. Appl. 509(1) (2022) 125946.
%
\bibitem{Ba1} M. Bachir, {\it Convex extension of lower semicontinuous functions defined on normal Hausdorff space}, J. Convex Anal. 27(3) (2020) 1033-1049.  
%
\bibitem{Ba2} M. Bachir, {\it A non-convex analogue to Fenchel duality} J. Funct. Anal. 181(2) (2001) 300-312.
%
\bibitem{COR} B. Cascales, J. Orihuela, M. Ruiz Galán, {\it Compactness, optimality, and risk}, Computational and analytical mathematics, 161-218, Springer Proc. Math. Stat. 50, Springer, New York (2013).
%
\bibitem{Cr1} B.D Craven,  {\it Mathematical Programming and Control Theory}, Chapman \& Hall London (1978).
%
\bibitem{Cr2}  B.D. Craven, B. Mond, {\it Transposition theorems for cone-convex functions}, SIAM J. Appl. Math. 24 (1973) 603-612.
%
\bibitem{DF} R. Deville, C. Finet, {\it An extension of Simons’ inequality and applications}, Rev. Mat. Univ. Complut. 14(1)  (2001), 95-104. MR1851724 (2002g:46014)
%
%
\bibitem{JON} V. Jeyakumar, W. Oettli, M. Natividad, {\it A solvability theorem for a class of quasiconvex mappings with applications to optimization}, J. Math. Anal. Appl. 179(2) (1993) 537-546.
%
\bibitem{Ga} A. Grothendieck, {\it Critères de compacité dans les espaces fonctionnels généraux}, Amer. J. Math. 74 (1952) 168-186.
%
\bibitem{Fk} K. Fan, {\it Minimax theorems}, Proceedings of the National Academy of Sciences of the United States of America 39 (1953) 42–47.
%
\bibitem{FGH} K. Fan, I. Glicksberg, A.J. Hoffman,  {\it Systems of inequalities involving convex functions}, Proc. Amer. Math. Soc. 8 (1957) 617-622.
%
%
\bibitem{Je1} V. Jeyakumar, {\it Convex-like alternative theorems and mathematical programming}, Optimization 16 (1985) 643-652.
%
%
\bibitem{JON} V. Jeyakumar, W. Oettli, M. Natividad, {\it A solvability theorem for a class of quasiconvex mappings with applications to optimization}, J. Math. Anal. Appl. 179(2) (1993) 537-546.
%
%
\bibitem{KO} K. Kivisoo, E. Oja, {\it Extension of Simons’ inequality} Proc. Am. Math. Soc. 133(12) (2005) 3485-3496.
%
%
\bibitem{Ld} D.  Luenberger, {\it Optimization by Vector Space Methods},  Wiley, New York (1969).
%
\bibitem{Ps} S. Paeck, {\it Convexlike and concavelike conditions in alternative, minimax, and minimization theorems}, J. Optim. Theory Appl. 74(2) (1992) 317-332.
%
\bibitem{Pt} V. Ptak, {\it Concerning spaces of continuous functions}, Czechoslovak Math. J. 5(80) (1955) 412-431.
%
%
\bibitem{Re} E. Reznichenko, {\it Extension of functions deﬁned on products of pseu- docompact spaces and continuity of the inverse in pseudocompact groups}, Topology and its Applications 59(3) (1994) 233-244. doi:https://doi.org/10.1016/0166-8641(94)90021-3.
%
\bibitem{Rt} T. W. Rishel, {\it Products of countably compact spaces} Proc. Amer. Math. Soc. 58 (1976) 329-330.
%
\bibitem{RG0} M. Ruiz Gal$\acute{\textnormal{a}}$n, {\it A theorem of the alternative with an arbitrary number of inequalities and quadratic programming}, Journal of Global Optimization, 69(2) (2017) 427-442.
%
\bibitem{RG1} M. Ruiz Gal$\acute{\textnormal{a}}$n, {\it The Gordan theorem and its implications for minimax theory}, Journal of Nonlinear and Convex Analysis 17(12) (2016) 2385-2405.
%
\bibitem{RG2} M. Ruiz Gal$\acute{\textnormal{a}}$n, {\it An intrinsic notion of convexity for minimax}, Journal of Convex Analysis 21(4) (2014)1105-1139.
%
\bibitem {Simons0} S. Simons, {\it Minimax and variational inequalities, are they of ﬁxed point or Hahn–Banach type? Game theory and mathematical economics},  North Holland (1981) 379-388.
%
\bibitem {Simons1} S. Simons, {\it Minimax theorems and their proofs, Minimax and applications}, Nonconvex Optim. Appl., vol. 4, Kluwer Acad. Publ., Dordrecht, (1995) 1-23.
%
\bibitem {Simons3} S. Simons, {\it A convergence theorem with boundary}, Paciﬁc J. Math. 40(3) (1972) 703-708.
%
\bibitem {St}  A. Stefanescu, {\it A theorem of the alternative and a two-function minimax theorem}, Journal of Applied Mathematics, 2 (2004) 169-177.
%
\bibitem {St1} A. Stefanescu, {\it Alternative and Minimax Theorems beyond Vector Spaces}, Journal of Mathematical Analysis and Applications, Volume 264(2) (2001) 450-464. 
%
\bibitem {Strm} R. M. Stephenson, Jr., {\it Pseudocompact spaces}, Trans. Amer. Math. Soc. 134(3) (1968) 437-448.  
%
\bibitem {Ts} S. Tomasek, {\it On a certain class of $bigwedge$-structures. I, II}, Czechoslovak Math. J. 20(95) (1970) 1-18; 19-33.
%
\bibitem{We} N. Weaver, {\it Lipschitz Algebras},  World Scientific, Second edition, New Jersey (2018).
%
\bibitem{Wrf} R. F. Wheeler, {\it Weak and Pointwise Compactness in the Space of Bounded Continuous Functions} Transactions of the American Mathematical Society, 266(2) (1981) 515-530.
\end{thebibliography}

\end{document}